\documentclass[11pt]{amsart}        % Your input file must contain these two
\usepackage[]{Mytheorems}
\usepackage{amssymb,amscd,amsthm,amsxtra}
\newtheorem{thm}{Theorem}[section]
\newtheorem{lem}[thm]{Lemma}

\newtheorem{prop}[thm]{Proposition}
\usepackage{psfrag}
\usepackage{mathrsfs,amscd}
\usepackage{graphicx}
\usepackage{epstopdf}

\usepackage{latexsym}
\input epsf

\usepackage{color}

\begin{document}
\title{Regularity of stationary  solutions to the  linearized Boltzmann equations }
\author[IC]{I-Kun Chen}
\address{Graduate School of Informatics,
Kyoto University, Yoshida-Honmachi,
Sakyo, Koto 6068501, Kyoto, Japan }
\email{ikun.chen@gmail.com}
\date{\today}
\maketitle
\begin{abstract}
We consider the regularity of  stationary solutions  to the linearized  Boltzmann equations in  bounded $C^1$ convex domains in $\mathbb{R}^3$  for gases with cutoff hard potential and cutoff Maxwellian gases. We prove that the stationary solutions solutions are  H\"{o}lder continuous with order $\frac1{2}^-$ away from the boundary provided the incoming data have the same  regularity.    The key idea is to partially transfer the regularity in velocity obtained by collision  to space through transport and collision.\end{abstract}
\section{introduction}

A thermal non-equilibrium stationary  solution, a resolution between the thermal dynamical tendency and boundary effects, reflects the complexity and richness of the Boltzmann equations. At linear and weakly nonlinear levels, the existence of such solutions  has been studied   by Guiraud  \cite{Guiraudlinear, Guiraudnonlinear} for convex domains and by  Esposito, Guo, Kim, and  Marra \cite{GuoKim} for nonconvex domains in $\mathbb{R}^3$. Regarding the regularity issue, as far as we know,  the best  result  is that the solution is continuous  away from the grazing set also in \cite{GuoKim}.  

For the corresponding time evolutionary weakly nonlinear problems, it was  studied by Kim, \cite{Kimdis}, that the discontinuity propagates from  the boundary to the interior of a non-convex domain on the tangent direction. The BV property of solutions in a non-convex domain was  studied by Guo, Kim, Tonon, and Trescases \cite{GKTT2}.  By the same authors,  regularity problem for convex domains was studied in \cite{GKTT}.  Especially, they established the weighted $C^1$ estimate for the specular reflection and diffuse reflection boundary conditions. However, the weighted $C^1$ norm grows severely with time, and therefore, could not give information to the stationary solution through large time behavior. This motivates us to look at the regularity to the stationary solution directly.   

The goal of this paper is to prove the interior regularity of the stationary linearized Boltzmann equation provided that the incoming data have the same regularity. After first laying  out the problem setting and definitions,  the precise statement of the main theorem will follow.

We consider the stationary linearized  Boltzmann equation 
\begin{equation}\label{SBE}
\zeta \cdot \nabla f(x,\zeta)=L(f),
\end{equation}
where $\zeta \in \mathbb{R}^3$ and $x\in\Omega$, a $C^1$ bounded convex domain in $\mathbb{R}^3$.     
The  linear collision operator $L$ here  is corresponding  to  the cutoff hard potential gases or cutoff Maxwellian gases, which are indicated by  $0\leq \gamma<1$  ( see \eqref{IP}). This cutoff was first introduced by Grad \cite{Grad},  and later the analysis was refined by Caflisch  \cite{Caflisch}.
Furthermore, we assume the cross section is a product of a  function of the length of relative velocity and a function of the deflecting angle as in \cite{ChenHsia}, i.e., we assume that the  cross section is nontrivial and satisfies\begin{equation}\label{IP}\begin{split}
 B(|\zeta_*-\zeta|,\theta)=|\zeta_*-\zeta|^{\gamma}\beta(\theta),\ \ 
0\leq\beta(\theta)\leq C\cos\theta\sin\theta.\end{split}
 \end{equation} Here, the cross section is for the binary collision operator
\begin{align}\label{QFF} Q(F,F)&=\int_{\mathbb{R}^3} \int_0^{2\pi}\int_0^{\frac{\pi}2}(F'F_*'-FF_*)B(|\zeta_*-\zeta|,\theta) d\theta d \epsilon d\xi_*\end{align}
before linearized around the standard Maxwellian $\pi^{-\frac32}e^{-|\zeta|^2}$.
  Under this assumption, $L$ has the following known properties (See \cite{Caflisch, ChenHsia, Grad}). $L$ can be decomposed into a multiplication operator and an integral operator: \begin{equation}
L(f)=-\nu(|\zeta|)f+K(f),
\end{equation}\\
where \begin{equation}K(f)(x,\zeta)=\int_{\mathbb{R}^3}k(\zeta,\zeta_*)f(x,\zeta_*)d\zeta_*\end{equation} is symmetric.  The explicit expression of $\nu$ is
\begin{equation}
\nu(|\zeta|)=\beta_0\int_{\mathbb{R}^3}e^{-|\eta|^2}|\eta-\zeta|^\gamma d\eta,
\end{equation}
where $\beta_0=\int_0^{\frac{\pi}2}\beta(\theta)d\theta$. Let $0<\delta<1$. The collision frequency $\nu(|\zeta|)$ and the collision kernel $k(\zeta,\zeta_*)$ satisfies\begin{align}&\nu_0(1+|\zeta|)^\gamma\leq \nu(|\zeta|)\leq \nu_1(1+|\zeta|)^\gamma,\\ \label{estimateK}
&|k(\zeta,\zeta_*)|\leq C_1|\zeta-\zeta_*|^{-1}(1+|\zeta|+|\zeta_*|)^{-(1-\gamma)}{e^{-{\frac{1-\delta}4}\left(|\zeta-\zeta_*|^2+(\frac{|\zeta|^2-|\zeta_*|^2}{|\zeta-\zeta_*|})^2\right)}},\\ \label{gradianK}&|\bigtriangledown_\zeta k(\zeta,\zeta_*)|\leq C_2\frac{1+|\zeta|}{|\zeta-\zeta_*|^{2}}(1+|\zeta|+|\zeta_*|)^{-(1-\gamma)}{e^{-{\frac{1-\delta}4}\left(|\zeta-\zeta_*|^2+(\frac{|\zeta|^2-|\zeta_*|^2}{|\zeta-\zeta_*|})^2\right)}}.
\end{align}
Here,  the constants $0<\nu_0<\nu_1$ may depend on the potential and  $C_1$ and $C_2$  may depend on $\delta$ and the potential. Related to the above estimates,  the following proposition from \cite{Caflisch} is crucial in our study.
\begin{prop} \label{cafdecay} For any $\epsilon, a_1, a_2>0$,
\begin{equation}
\Big|\int_{\mathbb{R}^3} \frac1{|\eta-\zeta_*|^{3-\epsilon}}e^{-a_1|\eta-\zeta_*|^2-a_2\frac{(|\eta|^2-|\zeta_*|^2)^2}{|\eta-\zeta_*|^2}}d\zeta_*\Big|\leq C_4 (1+|\eta|)^{-1}, \end{equation}
\noindent where $C_4$ may depend on $\epsilon, a_1,$ and $a_2$.
\end{prop}
\noindent

Suppose $x$ is a point inside  $\Omega$. We define $p(x,\zeta)$ to be the  boundary point that the backward trajectory from $x$ with velocity $\zeta$ touches. The corresponding traveling time is denoted by $\tau_-(x,\zeta)$. 
We view the integral operator $K$ as the source term:
\begin{equation}
\zeta  \nabla f(x,\zeta)+\nu(|\zeta|)f(x,\zeta)=K(f).
\end{equation}
Using the method of characteristics, we  derive the corresponding corresponding integral equation:\begin{equation}\label{inteq}\begin{split}
f(x,\zeta)&=f(p(x,\zeta),\zeta)e^{-\nu(|\zeta|) \tau_-(x,\zeta)}+\int_0^{\tau_-(x,\zeta)}e^{-\nu(|\zeta|) s}K(f)(x-\zeta s,\zeta)ds.
\end{split}\end{equation}
In this paper, we say $f$ is a solution to \eqref{SBE} if $f$ satisfies \eqref{inteq} almost everywhere.

 The boundary of $\Omega$ is denoted by $\partial\Omega$, and the outer normal is denoted by $\overrightarrow{n}$.  We define \begin{align}
%&\Gamma_+:= \{(x,\zeta)|x\in\partial\Omega, \zeta\cdot\overrightarrow{n}(x)>0\}, \\
&\Gamma_-:= \{(x,\zeta)|x\in\partial\Omega,\  \zeta\cdot\overrightarrow{n}(x)<0\}.
%\\&\Gamma_0:= \{(x,\zeta)|x\in\partial\Omega, \zeta\cdot\overrightarrow{n}(x)=0\}.
\end{align}
We consider norms as follows:
\begin{align}
\Vert g(\zeta)\Vert_{L^*_\zeta} &:=\left(\int_{\mathbb{R}^3} |g(\zeta)|^2\nu(|\zeta|)d\zeta \right)^\frac12,\\
\Vert f(x,\zeta)\Vert_{L^*_{x,\zeta}} &:=\left(\int_\Omega\int_{\mathbb{R}^3} |f(x,\zeta)|^2\nu(|\zeta|)d\zeta dx \right)^\frac12,\\
\Vert f(x,\zeta)\Vert _{L^\infty_xL^*_\zeta}&:=\sup_{x\in\Omega}\left(\int_{\mathbb{R}^3} |f(x,\zeta)|^2 \nu(|\zeta|)d\zeta\right)^\frac12,\\\Vert f(x,\zeta)\Vert _{L^\infty_{x,\zeta}}&:=\sup_{x\in\Omega, \zeta\in\mathbb{R}^3} |f(x,\zeta)|.\end{align}
The indices above denote the corresponding functional spaces.

The main conclusion in this paper is as follows. 
\begin{thm}Let  $0<\sigma<\frac12$ and $f\in L^*_{x,\zeta}$ be a stationary solution to the linearized Boltzmann equations \eqref{SBE},
\begin{equation*}
\zeta \cdot \nabla f(x,\zeta)=L(f),\end{equation*} in a $C^1$ bounded convex domain $\Omega$ in $\mathbb{R}^3$ for gases with cutoff hard potential or cutoff Maxwellian gases, i.e., the  corresponding cross sections  satisfy \eqref{IP} for $0\leq \gamma <1$. Suppose there exist $\phi(\zeta)\in L^*_{\zeta}$ and $M>0$ such that\begin{align}
|f(X,\eta)|&\leq \phi(\eta),\\
|f(X,\eta)-f(Y,\omega)|&\leq M\left( |\eta-\omega|^2+|X-Y|^2\right)^{\frac{\sigma}{2}} \end{align}
for any $(X,\eta)$, $(Y, \omega)\in\Gamma_-$.

Then, $f\in L^\infty_{x,\zeta}$. Furthermore,  there exists a constant $C_0$ depending only on $\Vert f\Vert_{L^\infty_{x,\zeta}}$, $\sigma$, $M$, $\Omega$, and the potential such that, for any  $x,\ y\in\Omega$ and $\zeta,\ \xi\in\mathbb{R}^3$,
\begin{equation} \label{holderd0}\begin{split}
&|f(x,\zeta)-f(y,\xi)|\leq  C_0(1+d_0^{-1})^3\left( |\zeta-\xi|^2+|x-y|^2\right)^{\frac{\sigma}{2}},\end{split}
\end{equation}
where $d_0$ is the distance of $x,\ y$ to $\partial \Omega$.
\end{thm} 

 Notice that $L^*_{x,\zeta}$ is the very functional space for the existence results in \cite{GuoKim, Guiraudlinear}. Also, notice that the inequality \eqref{holderd0} breaks down at boundary,  while some observations from the  explicit simple example  of the thermal transport problem may suggest so (see \cite{ CFLT, CLT,  Takata}). Although the hard sphere case is barely missed by the theorem, if we know solution is bounded a priori, we can recover the theorem for $\gamma=1$.

Our strategy is as follows.
We break the theorem down into three parts:  $L^*_{x,\zeta}\to L^\infty_xL^*_\zeta$, $ L^\infty_xL^*_\zeta\to L^\infty_{x,\zeta}$, and   $L^\infty_{x,\zeta}$ to H\"{o}lder continuity.  

In  one space dimensional case,  Golse and Poupaud used a bootstrap strategy to boost the integrability of solutions from $L^*_{x,\zeta}$ to $ L^\infty_xL^*_\zeta$ in finite  steps for Milne and Kramers' problems,  \cite{GP}. The key ingredient is to convert an estimate in velocity to a convolution in space, thanks to the transport and collision nature of Boltzmann equations and the simple geometry of one space dimension.  For the three dimensional problem, we manage to overcome the obstacle in geometry and extend the method to obtain $ L^\infty_xL^*_\zeta$ regularity.  Then, together with the assumption of H\"{o}lder continuity for the incoming data, we can show the solution is a bounded function in Section \ref{convolution}.  

It is well known that the integral operator $K$ can improve the regularity in velocity.  For the regularity in space, the H\"{o}lder continuity of $K(f)$ in space for solutions to one space dimensional equations is proved in \cite{IKC} for hard sphere gases, again thanks to the simple geometry of one space dimension. This observation also plays an important role in \cite{ChenHsia}. More precisely, in one space dimensional case, any two points  can be connected by a trajectory,  either forward
or backward, for almost all velocity. This is certainly not true in a three dimensional space.

 To carry out the H\"{o}lder continuity  for the three space dimensional problem we consider, we first iterate the integral equation once more:\begin{equation}\label{123def}\begin{split}
&f(x,\zeta)=f(p(x,\zeta),\zeta)e^{-\nu(|\zeta|) \tau_-(x,\zeta)}\\&
+\int_0^{\tau_-(x,\zeta)}\int_{\mathbb{R}^3}e^{-\nu(|\zeta|) s}k(\zeta,\zeta')e^{-\nu(|\zeta'|) \tau_-(x-\zeta s,\zeta')}f\big(p(x-\zeta s,\zeta'),\zeta'\big)d\zeta'ds\\
&+\int_0^{\tau_-(x,\zeta)}\int_{\mathbb{R}^3}\int_0^{\tau_-(x-\zeta s,\zeta')}e^{-\nu(|\zeta|) s}k(\zeta,\zeta')e^{-\nu(|\zeta'|) t}K(f)(x-\zeta s-\zeta't,\zeta')dtd\zeta'ds\\&=:I+II+III.
\end{split}\end{equation}

The $I$ and $II$ can preserve the regularity of $\Gamma_-$ due to the nature of the transport equation.  The key step is to prove the increasing of regularity in $III$.

The smoothing effect of $K$ in velocity has been studied, for example, in  \cite{IKC, ChenHsia, LiuYu1}. The key observation in this study is that, in $III$,  the regularity in velocity of $K(f)$ can be partially transferred   to space due to transport and collision.  In order to connect velocity to space, in Section \ref{Mix}, we  change  coordinates in the inner integral in $III$ to space over $\Omega$ and length of velocity over $\mathbb{R}^+$.  Due to singularity of the integrand in that formula, we are not able to push the regularity in space to differentiable.  However, using the fact that angle difference from an observer toward two different points, the  parallax, becomes smaller when the observer becomes farer away and balance the contribution from nearby by the smallness of domain of integration, we are able to obtain local H\"{o}lder continuity of $III$ to the order of $\frac12$ . 

  Readers familiar to the time evolutionary Boltzmann equations in the whole Euclidean space  may find an analogy to the Mixture Lemma studied by Liu and Yu \cite{LiuYu1, LiuYu2}, and later extend by  Kuo, Liu, and Noh \cite{KuoLiuNoh}, and  Wu \cite{ WuKungChien}, that is,  the regularity in velocity can be transfer to space through collision and transport.    This kind of idea can be traced back to   the famous Averaging Lemma by Golse, Perthame, and Sentis \cite{GolseAve}.  
In the context of stationary solution in a convex domain, despite of the same spirit, the new technique  to carry out the mixture effect as explained above is quite different. 
   
The plan of this paper is as follows. In Section \ref{Mix}, we study the key mixture effect. Then,   we investigate the geometric properties of a convex domain in Section \ref{convexdomain}, where we  show the local H\"{o}lder continuity of $III$. In section \ref{Boundaryreg}, we  show that the regularity of the boundary data are preserved by first and second iterations of the transport equation, $I$ and $II$. Eventually, we return to increasing of integrability from $L^*_{x,\zeta}$ to $L^\infty_{x,\zeta}$ in section \ref{convolution}  and complete our proof.

\section{\label{Mix} Gaining regularity from collision and transport}
We will elaborate the smoothing effect due to the combination of collision and transport in this section. We first  convert the formula by interplaying between velocity and space.
 We  change $\zeta'$ to the spherical  coordinates so that
\begin{equation}
\zeta'=(\rho\cos\theta,\rho\sin \theta\cos\phi,\rho\sin \theta\sin \phi).
\end{equation} 
Also, we change traveling time to the traveling distance:
\begin{equation}
r=\rho t.
\end{equation}
Let $\hat{\zeta'}=\frac{\zeta'}{|\zeta'|}.$
Then, 
\begin{equation}
\begin{split}
III=&\int_0^{\tau_-(x,\zeta)}e^{-\nu(|\zeta|) s}\int_0^\infty\int_0^{\pi}\int_0^{2\pi}\int_0^{|\overline{xp(x-s\zeta,\zeta')}|}\\&\quad  k(\zeta,\zeta')e^{-\frac{\nu(\rho) }\rho r}K(f)(x-\zeta s-\hat{\zeta'}r,\zeta')\rho\sin\theta drd\phi d\theta d\rho ds\\=&: \int_0^{\tau_-(x,\zeta)}e^{-\nu(|\zeta|) s}G(x-\zeta s,\zeta)ds.
\end{split}
\end{equation}
 Notice that  we can parametrize $\Omega$ by $\theta$, $\phi$, and $r$, thanks to the convexity.  Therefore, by regrouping the integrals, we can change the formulation to contain an  integral over space.  Let $x_0=x-\zeta s$ and $y=x-\zeta s-\hat{\zeta'}r$. We have

\begin{equation}\label{Gformula}
\begin{split}
&G(x_0,\zeta)=\\&\quad\int_0^\infty\int_{\Omega} k(\zeta,\frac{(x_0-y)\rho}{|x_0-y|})e^{-\nu(\rho)\frac{|x_0-y|}{\rho}}K(f)(y, \frac{(x_0-y)\rho}{|x_0-y|})\frac{\rho}{|x_0-y|^2}dyd\rho.
\end{split}
\end{equation}
Notice that if we differentiate the above formula directly respect to $x_0$, it will become not integrable in $y$. However, we can still obtain a lower regularity from this expression. 
The main purpose of this section  is to prove the following lemma.
\begin{lem}\label{MixLemma}
Suppose $f(x,\zeta)\in L^{\infty}_{x,\zeta}$ is a solution to \eqref{SBE} and $x_0$, $x_1\in \Omega$.
Then,  there exist $C_3$ only depending on the $\Omega$ and the potential  such that \begin{equation}
|G(x_0,\zeta)-G(x_1,\zeta)|\leq C_3\Vert f\Vert_{L^\infty_{x,\zeta}}|x_0-x_1|^{\frac12}.
\end{equation}
\end{lem}
We will follow the strategy explained in the introduction.
As mentioned, the smoothing effect of $K$ in velocity is a well known fact. It was used  to study  Green's functions of the Boltzmann equations in \cite{LiuYu1}.
In \cite{IKC}, the locally H\"{o}lder continuity of $K(f)$ in velocity  has been shown for $f\in L^2_\zeta$  for hard sphere gases.  Here, we will use an estimate appeared in \cite{ChenHsia}. \begin{prop}\label{holdervelocity}
Suppose  $1\leq p \leq \infty$ and $f\in L^p_\zeta$.  Then, there exists $C_p$ depending only on $p$ and the potential such that  
\begin{equation}
\Vert\bigtriangledown_\zeta K(f)(\zeta)\Vert_{L^p_\zeta}\leq C_p\Vert f\Vert_{L^p_\zeta} 
.\end{equation}
\end{prop}
Using \eqref{gradianK}  and Proposition  \ref{cafdecay}  we can conclude that both $\Vert \bigtriangledown_\zeta k(\zeta,\zeta_*) \Vert_{L^\infty_\zeta L^1_{\zeta_*}}$ 
and $\Vert \bigtriangledown_\zeta k(\zeta,\zeta_*) \Vert_{L^\infty_{\zeta_*} L^1_\zeta}$ are bounded. Therefore, we can prove the Proposition \ref{holdervelocity} by an argument similar to the proof of Young's inequality.

%Recall that \begin{equation}
%\Vert\partial_{\zeta_i}K(f)\Vert_p\leq \Vert f\Vert_p,
%\end{equation}
%where $1\leq p \leq \infty$.
%From Morrey's inequality, we can conclude, for $p>3$,
%\begin{equation}
%\Vert K(f)\Vert_{C^{0,\beta}}\leq\Vert f\Vert_p,
%\end{equation} 
%where $\beta=1-\frac3p$.
In addition to the above formula, we also need an estimate for  the parallax  to transfer the regularity from velocity to space.

\begin{prop}[Estimate for the parallax]\label{anglediff}Let $x_0$, $x_1$, and $y\in \mathbb{R}^3$ and $0<a<1$. We denote $|x_0-x_1|=d$ and  $\angle {x_0yx_1}=\theta$. If  $|y-x_0|>2d^a$ and $d\leq 1$, then \begin{equation}\theta<\frac{\pi }4d^{1-a}.\end{equation}
\end{prop}

\begin{proof}
We only need to work on the geometry on the plane passing $x_0$, $x_1$, and $y$ as showing in the Figure \ref{figsmall}.
\begin{figure}\label{figsmall}
\centering
\includegraphics[scale=.6]{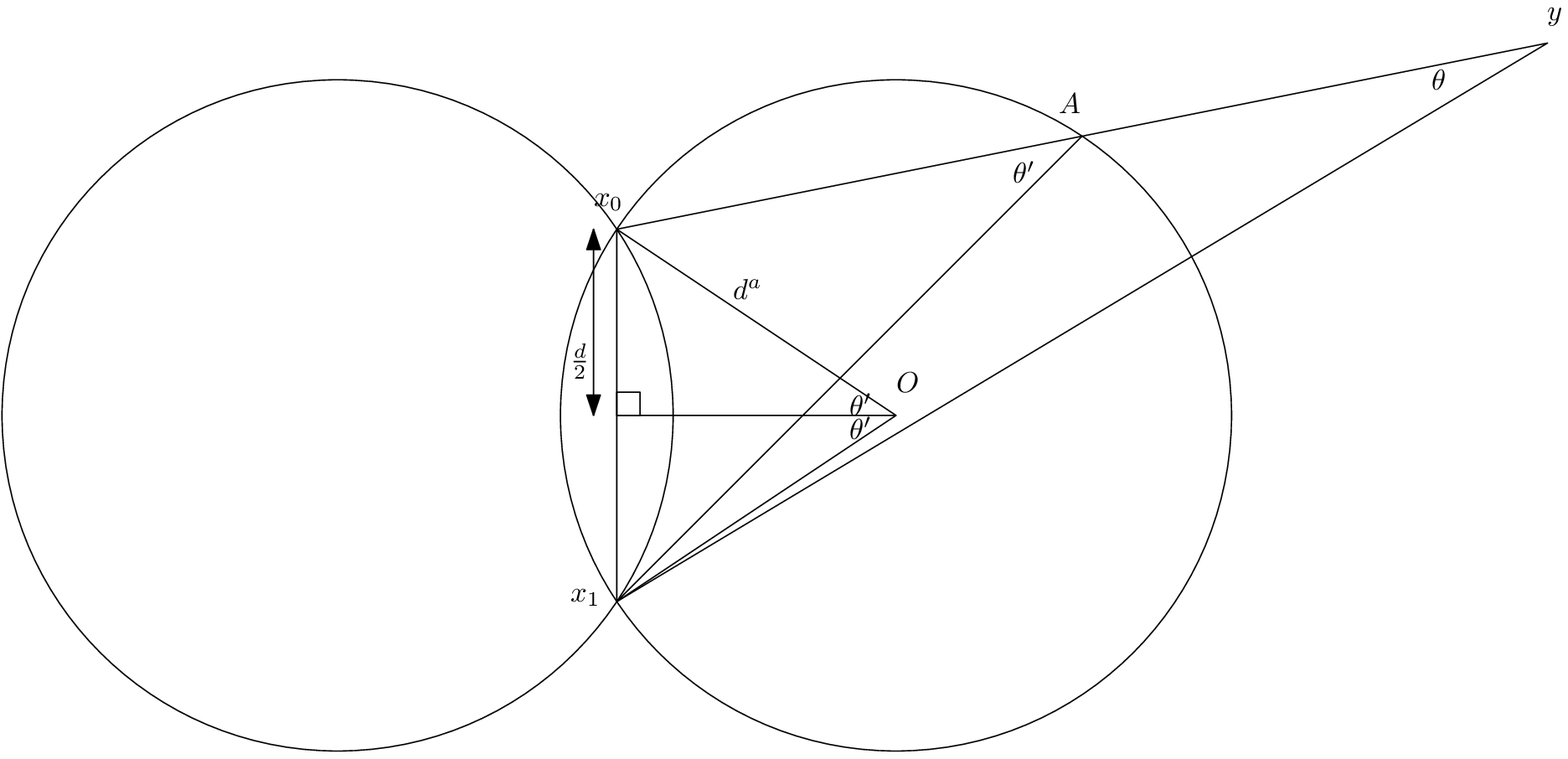}
\caption{}
\end{figure}
On this plane, we draw the two circles passing $x_0$ and $x_1$ with radius $d^a$. We observe  that $y$ is outside of the union of two larger arcs with $x_0$ and $x_1$ as end points. Let $A$ be the intersection point between $\overline{x_0y}$  and one of the larger arcs. We denote the center of the circle containing this arc as $O$. We have 
\begin{equation}
\theta=\angle{x_0yx_1}<\angle{x_0Ax_1}=:\theta'.\end{equation}
On the other hand, \begin{equation}
2\angle{x_0Ax_1}=\angle{x_0Ox_1}<\frac{\pi}2.   \end{equation}
We observe
\begin{equation}
\sin(\theta')d^a=\frac{d}2.
\end{equation} 

Applying Jordan's inequality, we have
\begin{equation}
\theta<\theta'\leq\frac{\pi}2\sin\theta'=\frac{\pi}4d^{1-a}.    
\end{equation}
\end{proof}

Now, we are ready to prove Lemma \ref{MixLemma}.
\begin{proof}[Proof of Lemma \ref{MixLemma}]
We add and subtract  to obtain 
 \begin{equation}
\begin{split}
&\left|G(x_0,\zeta)-G(x_1,\zeta)\right|\leq
\bigg|\int_0^\infty\int_{\Omega} k(\zeta,\frac{(x_0-y)\rho}{|x_0-y|})\frac{\rho e^{-\nu(\rho)\frac{|x_0-y|}{\rho}}}{|x_0-y|^2}\\&\quad\cdot\Big[K(f)(y, \frac{(x_0-y)\rho}{|x_0-y|})-K(f)(y, \frac{(x_1-y)\rho}{|x_1-y|})\Big]dyd\rho\bigg|\\&+\bigg|\int_0^\infty\int_{\Omega} K(f)(y, \frac{(x_1-y)\rho}{|x_1-y|})\\&\quad\cdot\Big[k(\zeta,\frac{(x_0-y)\rho}{|x_0-y|})\frac{\rho e^{-\nu(\rho)\frac{|x_0-y|}{\rho}}}{|x_0-y|^2}-k(\zeta,\frac{(x_1-y)\rho}{|x_1-y|})\frac{\rho e^{-\nu(\rho)\frac{|x_1-y|}{\rho}}}{|x_1-y|^2}\Big]dyd\rho\bigg|\\&=:G_K+G_O.
\end{split}
\end{equation}
We first deal with $G_K$.  We use $B(x,r)$ to  denote the open ball centered at $x$ with radius $r$. 
We break the domain of integration into two,  $\Omega_1:= \Omega \setminus B(x_0, 2d^a)$ and $\Omega_2:=\Omega\cap B(x_0,2d^a)$  for $0<a<1$ to be determinate later and name the corresponding integrals as $A_1$ and $A_2$ respectively.  For $A_1$ , we will use the smoothness of $K(f)$ in velocity, while, for $A_2$, we use the smallness of domain. Notice that, from Proposition \ref{holdervelocity},  $K(f)$ is Lipschitz continuous in velocity. 
\begin{equation}
|K(f)(y,\zeta_1)-K(f)(y,\zeta_2)|\leq C\Vert f\Vert_{L^\infty_{x,\zeta}}|\zeta_1-\zeta_2|.
\end{equation}
 Together with Proposition \ref{anglediff}, we have
 \begin{equation}\begin{split}
 &\left|K(f)(y, \frac{(x_0-y)\rho}{|x_0-y|})-K(f)(y, \frac{(x_1-y)\rho}{|x_1-y|})\right|\\&\leq C \Vert f\Vert_{L^\infty_{x,\zeta}}\left|\frac{(x_0-y)\rho}{|x_0-y|}-\frac{(x_1-y)\rho}{|x_1-y|}\right|\leq  C \Vert f\Vert_{L^\infty_{x,\zeta}}\rho \theta\\&\leq   C \Vert f\Vert_{L^\infty_{x,\zeta}}\rho d^{1-a}.
 \end{split}\end{equation}
 Therefore,
 
\begin{equation}\begin{split}&|A_1|\leq C\Vert f\Vert_{L^\infty_{x,\zeta}}\int_0^\infty\int_{\Omega_1} \big|k(\zeta,\frac{(x_0-y)\rho}{|x_0-y|})\big|\frac{\rho e^{-\nu(\rho)\frac{|x_0-y|}{\rho}}}{|x_0-y|^2}d^{(1-a)}\rho dyd\rho\\&
\leq C d^{(1-a)}\Vert f\Vert_{L^\infty_{x,\zeta}}\int_0^\infty\int_0^\pi\int_0^{2\pi}\int_{2d^{a}}^R  \big|k(\zeta,\frac{(x_0-y)\rho}{|x_0-y|})\big|\rho^{2}\sin\theta drd\phi d\theta d\rho\\
&\leq C d^{(1-a)}\Vert f\Vert_{L^\infty_{x,\zeta}}\int_{2d^{a}}^R \int_{\mathbb{R}^3}|k(\zeta,\zeta')|d\zeta'd r\\&\leq C d^{(1-a)}\Vert f\Vert_{L^\infty_{x,\zeta}}R,\end{split}\end{equation} 
where $R$ is the diameter of $\Omega$. Notice that we changed back the coordinates in order to take advantage of the integrability of $k(\zeta,\zeta')$ in the above inequality. 

On the other hand, using \eqref{estimateK},  we have
\begin{equation}\label{BdLinf}
\Vert K(f)(x,\zeta)\Vert_{L^\infty_{x,\zeta}}\leq C\Vert f\Vert_{L^\infty_{x,\zeta}}.
\end{equation}
Therefore,
\begin{equation}\begin{split}&|A_2|\leq C\Vert f \Vert_{L^\infty_{x,\zeta}}\int_0^\infty\int_{B(x_0,2d^a)}\big| k(\zeta,\frac{(x_0-y)\rho}{|x_0-y|})\big|e^{-\nu(\rho)\frac{|x_0-y|}{\rho}}\frac{\rho}{|x_0-y|^2}dyd\rho\\&\leq C\Vert f \Vert_{L^\infty_{x,\zeta}}\int_0^\infty\int_0^\pi\int_0^{2\pi}\int_0^{2d^{a}}\big| k(\zeta,\frac{(x_0-y)\rho}{|x_0-y|})\big|e^{-\nu(\rho)\frac{|x_0-y|}{\rho}}\rho\sin\theta dr d\phi d\theta d\rho
\\&\leq C\Vert f \Vert_{L^\infty_{x,\zeta}}\int_0^{2d^{a}} \int_{\mathbb{R}^3}|k(\zeta,\zeta')|\frac1{|\zeta'|}d\zeta'd r\leq C\Vert f \Vert_{L^\infty_{x,\zeta}} d^a.\end{split}\end{equation}  
To optimize the estimate, we choose $a=
\frac12$, which gives the desired estimate for $G_K$.

Now, we proceed to estimate $G_O$.
We divide the domain of integration into two, $\Omega_3:=\Omega \setminus B(x_0,2d)$ and $\Omega_4:=\Omega\cap B(x_0,2d)$, and name the corresponding integrals as $A_3$ and $A_4$ respectively.

We first deal with $A_3$.   Let $\tilde{x}(u)=x_1+(x_0-x_1)u$.  We have
\begin{equation}\begin{split}
&|A_3|=\\&\left|\int_0^\infty\int_{\Omega_3} \int_0^1\frac{d}{du}[k(\zeta,\frac{(\tilde{x}(u)-y)\rho}{|\tilde{x}(u)-y|})\frac{\rho e^{-\nu(\rho)\frac{|\tilde{x}(u)-y|}{\rho}}}{|\tilde{x}(u)-y|^2}]duK(f)(y, \frac{(x_1-y)\rho}{|x_1-y|})dyd\rho\right|.
\end{split}\end{equation}
By direct calculation and using \eqref{estimateK}, \eqref{gradianK}, \eqref{BdLinf}, together with Proposition \ref{cafdecay}, we obtain
 \begin{equation}
\begin{split}
&|A_3|
\\&\leq C|x_1-x_0|\Vert f\Vert_{L^\infty_{x,\zeta}}\int_0^1\int_0^\infty\int_{\Omega_3} e^{-\frac18\left[|\zeta-\frac{(\tilde{x}(u)-y)\rho}{|\tilde{x}(u)-y|}|^2+(\frac{|\zeta|^2-\rho^2}{|\zeta-\zeta'(\rho,\tilde{x}(u),y)|})^2\right]}\\&\quad \cdot\frac{1}{|\tilde{x}(u)-y|^3}\left[\frac{\rho^2(1+\rho)}{|\zeta-\frac{(\tilde{x}(u)-y)\rho}{|\tilde{x}(u)-y|}|^2}+\frac{\rho}{|\zeta-\frac{(\tilde{x}(u)-y)\rho}{|\tilde{x}(u)-y|}|}\right]dyd\rho du
\\&\leq Cd\Vert f\Vert_{L^\infty_{x,\zeta}}\int_0^1\int_{d}^R\int_0^\pi\int_0^{2\pi} \int_0^\infty\frac{e^{-\frac18\left[|\zeta-\frac{(\tilde{x}(u)-y)\rho}{|\tilde{x}(u)-y|}|^2+(\frac{|\zeta|^2-\rho^2}{|\zeta-\zeta'(\rho,\tilde{x}(u),y)|})^2\right]}}{r}\\&\quad \cdot\left[\frac{\rho^2(1+\rho)}{|\zeta-\frac{(\tilde{x}(u)-y)\rho}{|\tilde{x}(u)-y|}|^2}+\frac{\rho}{|\zeta-\frac{(\tilde{x}(u)-y)\rho}{|\tilde{x}(u)-y|}|}\right]\sin\theta d\rho d\phi d\theta dr du
\\&\leq Cd\Vert f\Vert_{L^\infty_{x,\zeta}}\int_0^1\int_{d}^R\int_{\mathbb{R}^3} e^{-\frac18\left[|\zeta-\zeta'|^2+(\frac{|\zeta|^2-|\zeta'|^2}{|\zeta-\zeta'|})^2\right]}\\&\quad \cdot\left[\frac{1+|\zeta|+|\zeta-\zeta'|}{|\zeta-\zeta'|^2}+\frac{1}{|\zeta-\zeta'||\zeta'|}\right]d\zeta'\frac1rdrdu
\\&\leq Cd(1+|\ln d|)\Vert f\Vert_{L^\infty_{x,\zeta}}.
\end{split}
\end{equation}
Notice that we applied\begin{equation}|\tilde{x}(u)-y|\geq|x_0-y|-|\tilde{x}(u)-x_0|\geq d
\end{equation}
in the above inequality.

On the other hand, notice that  $\Omega_4\subseteq B(x_0, 2d) \subseteq B(x_1,3d)$. We can obtain
\begin{equation}
|A_4|\leq Cd \Vert f\Vert_{L^\infty_{x,\zeta}}.
\end{equation}
Combining all the estimates above, we conclude the lemma.

\end{proof}
In the end of this section, we will discuss the regularity of $G$  in velocity.  Notice that \begin{equation}
\left|\int_0^{\tau_-(x_0,\zeta')}e^{-\nu t}K(f)(x_0-\zeta't,\zeta')dt\right|< C\Vert f\Vert_{L^\infty_{x,\zeta}}.
\end{equation}
Applying  Proposition \ref{holdervelocity}, we can conclude
\begin{prop}\label{Gzetadiff}
\begin{equation}
|G(x_0,\zeta_1)-G(x_0,\zeta_2)|\leq C_5\Vert f\Vert_{L^\infty_{x,\zeta}}|\zeta_1-\zeta_2|,
\end{equation}
where the constant $C_5$ depending only on the potential.
\end{prop}

\section{Properties of a convex domain \label{convexdomain}}
In this section, we will discuss some properties of a convex domain and then apply to the proof of the regularity of $III$. These properties are all based on the same observation, that is,   a line connecting an interior point and a boundary point can not have  too small an angle to the tangent plane in a convex domain if the point is away from the boundary.  The following proposition and its proof  serve as an examples. 
\begin{prop}\label{G12}
Suppose  $\Omega$ is a bounded $C^1$ convex domain in $ \mathbb{R}^3$ and $x, \ y \in\Omega$.  Then, there exists a constant  $C_6$ only depending on $\Omega$ such that\begin{align}
&|p(x,\zeta)-p(y,\zeta)| \leq  C_6(1+d_0^{-1})|x-y|\label{geometry1},\\
&|\tau_-(x,\zeta)-\tau_-(y,\zeta)|\leq  C_6(1+d_0^{-1})\frac{|x-y|}{|\zeta|}\label{geometry2}.
\end{align}
\end{prop}
\begin{proof} We will work on the  geometry of the plane passing two points $x$ and $y$ and containing the vector $\zeta$  as showing in Figure \ref{figconvex}.
\begin{figure}
\centering
\includegraphics[scale=.6]{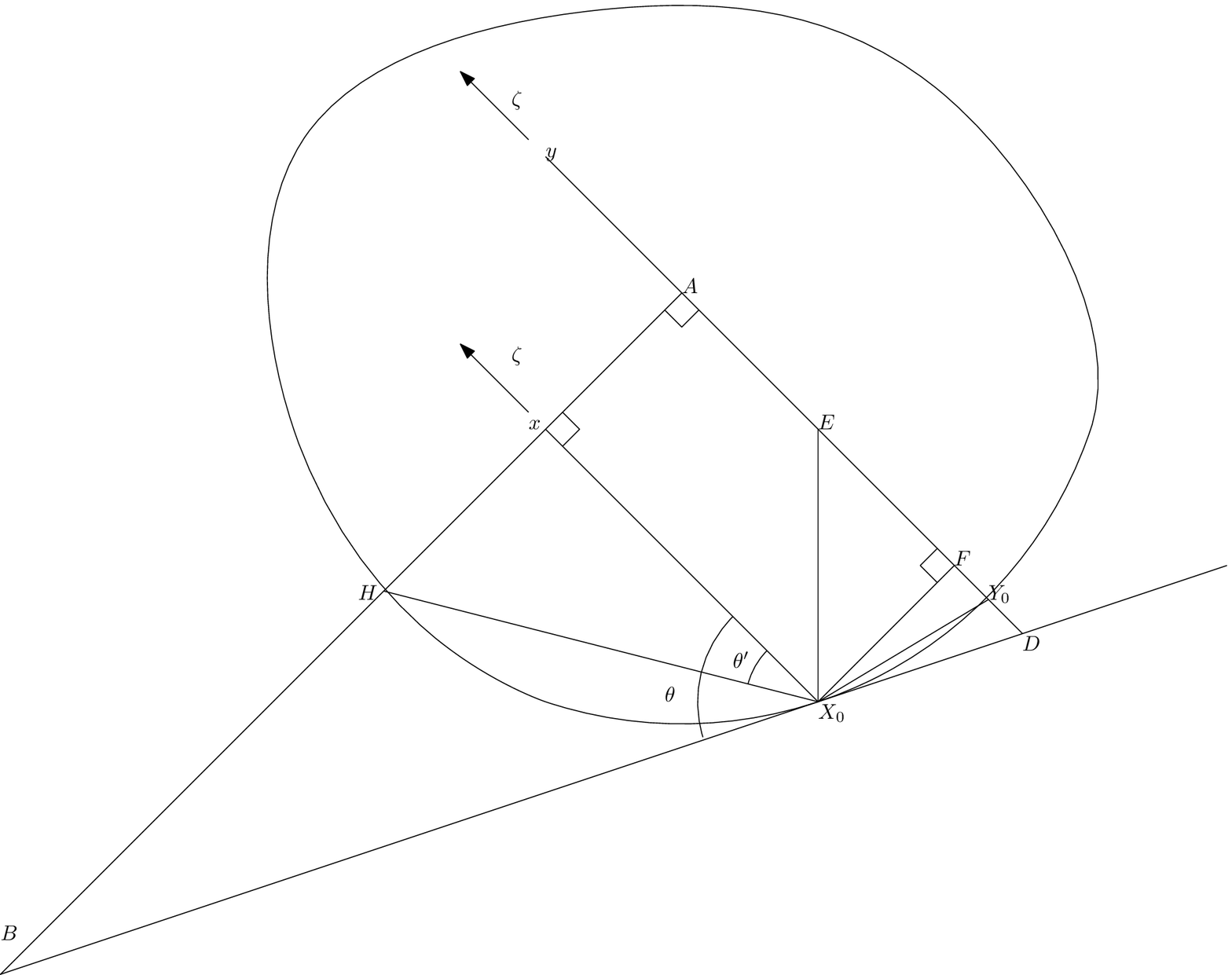}
\caption{}\label{figconvex}
\end{figure}Without lost of generlity, we may assume  $\tau_-(y,\zeta)\geq\tau_-(x,\zeta)$.    We note $X_0=p(x,\zeta)$ and $Y_0=p(y,\zeta)$. If $y-x\parallel \zeta$, then $X_0=Y_0$ and $|\tau_-(x,\zeta)-\tau_-(y,\zeta)|=\frac{|x-y|}{|\zeta|}$ and the lemma holds. Otherwise,   let $\zeta_\perp$ be the projection of the vector $x-y$ on the subspace perpendicular to $\zeta$. The line passing $x$ with direction $\zeta_\perp$ is denoted by $L_\perp$, which intersects the line $\overleftrightarrow{yY_0}$ at a point $A$.   Suppose   $L_\perp$ intersects the tangent plane of $\Omega$ passing $X_0$ at  a point $B$. Then, the line segment $\overline{xB}$ intersects $\partial\Omega$ at a point $H$, because of the convexity.    $\overleftrightarrow{BX_0}$ and $\overleftrightarrow{yY_0}$ intersect at $D$.  Let $E=X_0+(y-x)$ and $F=X_0+(A-x)$.  Define $\theta: =\angle ADB=\angle xX_0B$.  Notice that $0<\theta\leq \frac{\pi}{2}$ because  $\overleftrightarrow{Ax}\perp \overleftrightarrow{AD}$.  We observe $0<\theta':=\angle xX_0H<\theta$ because of the convexity. Notice that \begin{equation}
\sin\theta'=\frac{|\overline{xH}|}{|\overline{X_0H}|}\geq\frac{d_0}{R}.
\end{equation}
Therefore, \begin{equation}|\overline{X_0D}|=\frac{|\overline{X_0F}|}{\sin\theta}\leq \frac{R}{d_0}|\overline{X_0F}|\leq \frac{R}{d_0}|x-y|.
\end{equation}
Notice that $|\zeta|\cdot|\tau_-(x,\zeta)-\tau_-(y,\zeta)|=|\overline{EY_0}|<|\overline{ED}|$. Then, applying triangular inequality for  $ \bigtriangleup  EDX_0$, we can conclude \eqref{geometry2}. Again, using triangular inequality for  $ \bigtriangleup  EY_0X_0$, we obtain  \eqref{geometry1}. Therefore, we proved the proposition for the case $L_\perp$  intersects with the tangent plane passing $x$.  For the other case, i.e., if $L_\perp$ is parallel to the tangent plane passing $X_0$, we can easily see that the proposition holds.
\end{proof}
We use  $d(z, \partial \Omega)$ to denote the distance from $z$ to $\partial \Omega$.
By a similar argument, we  have the following Proposition to be used  the next section
\begin{prop}\label{distanceXzest}
With the same assumption in  Proposition \ref{G12}. Suppose $z\in \overline{xX_0}$. Then, 
\begin{equation}
|\overline{zX}|\leq \frac{R}{d_0}d(z, \partial \Omega).\end{equation}
\end{prop} 
For two trajectories passing the same point, we have the following proposition.
\begin{prop}\label{Geotheta}
Suppose $\Omega$ is a bounded $C^1$ convex domain and  $\zeta_1, \zeta_2 \in \mathbb{R}^3$ and $x \in \Omega$. Let $P_1=p(x,\zeta_1)$ and $P_2=p(x,\zeta_2)$. Let $d_0$ be the distance between $x$ and $\partial \Omega$ and $\theta$ be the angle between $\zeta_1$ and $\zeta_2$.  Then, there exists a constant $C_7$ depending only on $\Omega$ such that
\begin{align}
&|P_1-P_2|\leq C_7(1+d_0^{-1})\theta,
\\&\big||\overline{xP_1}|-|\overline{xP_2}|\big|\leq C_7(1+d_0^{-1})\theta.
\end{align}\end{prop}
\begin{proof}
When $\theta$ is big, the lemma is true because $\Omega$ is bounded. Thus, we only need to deal with the case $\theta<1$.    We are going to work on the plane geometry of the plane passing $x$, $P_1$, and $P_2$.  Without lost of generality, we can assume $|\overline{xP_2}|\geq|\overline{xP_1}|$ . Let $A'$ be the projection of $P_1$ on $\overline{xP_2}$. Notice that \begin{equation}|\overline{P_1A'}|\leq R\sin\theta\leq R\theta.
\end{equation} On the plane we considered, the line passing $x$ and perpendicular to $\overline{xP_2}$ intersects $\Omega$ at two points. The one at the same side with $P_1$  is denoted by $B'$. Because of the convexity, we  observe\begin{equation}\frac{d_0}{R}\leq \sin \angle B'P_2x\leq \sin\angle P_1P_2x.
\end{equation}
Therefore, 
\begin{equation}
|\overline{P_1P_2}|=\frac{|\overline{P_1A'}|}{\sin\angle P_1P_2x}\leq \frac{R^2}{d_0} \theta.\end{equation}
Also,
\begin{equation}
\big||\overline{xP_1}|-|\overline{xP_2}|\big|\leq |\overline{A'P_2}|\leq |\overline {P_1P_2}|\leq \frac{R^2}{d_0} \theta.
\end{equation}
Hence, we complete the proof.
\end{proof}

%\begin{proof}[Proof of Lemma \ref{G34}]
%We  divide the $\partial \Omega$  by a plane passing $y$ parallel  to the tangent plane passing $x$ and name the one includes $x$ as $\partial \Omega_x$. Notice all possible $p(y,\zeta)$ are one $\partial \Omega_x$. We can estimate the size of  $\partial \Omega_x$ using the uniform strict convexity and conclude the result.
%\end{proof}
Now we are ready to discuss the regularity of $III$.
\begin{prop} Let $\Omega$ be a bounded $C^1$ domain in $\mathbb{R}^3$.  $III$ is defined in \eqref{123def} for $f\in L^\infty_{x,\zeta}$ that solves \eqref{SBE}.  Suppose $x$ and $y $ are interior points of $\Omega$ and $\zeta_1,\zeta_2\in \mathbb{R}^3$. Then, there exists $C_8$ depending only on  $\Omega$ and the potential of the gases, such that \begin{equation}| III(x,\zeta_1)-III(y,\zeta_2)|\leq  C_8(1+d_0^{-1})^3\Vert f\Vert_{L^{\infty}_{x,\zeta}}(|x-y|^2+|\zeta_1-\zeta_2|^2)^{\frac14}, \end{equation} where $d_0$ is the distance of $x$ and $y$ to $\partial\Omega$.

%If $(x,\zeta)\in\Gamma_-$, then\begin{equation}| III(y,\zeta)|\leq C |x-y|^\frac12\end{equation}
\end{prop}
\begin{proof}
Since $III$ is bounded, the cases for $|x-y|>1$ or $|\zeta_1-\zeta_2|>1$ are trivial. We only consider the case $|x-y|\leq1$ and $|\zeta_1-\zeta_2|\leq1$.
Let  $\bar{\zeta_1}=\frac{|\zeta_2|}{|\zeta_1|}\zeta_1$ and $\zeta\in\mathbb{R}^3$. We break the estimate into three parts:
\begin{align}& \label{III1}| III(x,\zeta)-III(y,\zeta)|\leq  C(1+d_0^{-1})^2\Vert f\Vert_{L^{\infty}_{x,\zeta}}|x_0-x_1|^{\frac12}, \\
&| III(x,\bar{\zeta_1})-III(x,\zeta_2)|\leq  C(1+d_0^{-1})^3\Vert f\Vert_{L^{\infty}_{x,\zeta}}|{\zeta_1}-\zeta_2|^{\frac12},\label{III2}\\&
| III(x,\bar{\zeta_1})-III(x,\zeta_1)|\leq  C(1+d_0^{-1})\Vert f\Vert_{L^{\infty}_{x,\zeta}}|\zeta_1-\zeta_2|^{\frac12}.\label{III3}\end{align} 
 We start with \eqref{III1}.  
 We will assume $\tau_-(y,\zeta)\geq\tau_-(x,\zeta)$ and  use  Proposition \ref{G12} and the notations in its proof.  When  $|x-y|<\frac{d_0}{2},$
   \begin{equation}\label{IIIx}\begin{split}
&| III(x,\zeta)-III(y,\zeta)|\leq\\&\left|\int_0^{\frac{|x-X_0|}{|\zeta|}}e^{-\nu s}\big[G(x-s\zeta,\zeta)-G(y-s\zeta,\zeta)\big]ds\right|+\left|\int_{\tau_-(x,\zeta)}^{\tau_-(y,\zeta)}e^{-\nu s}G(y-s\zeta,\zeta)ds\right|\\&\leq  C\Vert f\Vert_{L^{\infty}_{x,\zeta}}|x-y|^{\frac12}+{C}{(1+d_0^{-1})}\frac{|x-y|}{|\zeta|}e^{-\nu_0\frac{d_0}{2|\zeta|}}\Vert f\Vert_{L^{\infty}_{x,\zeta}}\\&\leq  C(1+d_0^{-1})^2\Vert f\Vert_{L^{\infty}_{x,\zeta}}|x_0-x_1|^{\frac12}.
\end{split}\end{equation}
Notice that we used the fact $\frac1{|\zeta|}e^{-\nu_0\frac{d_0}{2|\zeta|}}\leq \frac{C}{\nu_0d_0}$ and $|x-y|\leq1$ in the last inequality above.
When $|x-y|\geq \frac{d_0}{2|\zeta|},$\begin{equation}\label{Bdistanceest}\begin{split}
|III|&\leq C\Vert f\Vert_{L^{\infty}_{x,\zeta}}\\&\leq C\Vert f\Vert_{L^{\infty}_{x,\zeta}}\frac{|x-y|^{\frac12}}{d_0^\frac12}\\&\leq   C(1+d_0^{-1})^2\Vert f\Vert_{L^{\infty}_{x,\zeta}}|x_0-x_1|^{\frac12}.\end{split}
\end{equation}

%If $(x,\zeta)\in\Gamma_-$, applying Lemma \ref{G34}, we have
%\begin{equation}
%| III(y,\zeta)|=|\int_0^{\tau_-(y,\zeta)}e^{-\nu s}G(y-s\zeta,\zeta)ds|\leq C\Vert f\Vert_{L^{\infty}_xL^*_\zeta}|\tau_(y,\zeta)|\leq C\Vert f\Vert_{L^{\infty}_xL^*_\zeta}\frac{|x-y|^{\frac12}}{|\zeta|}\end{equation}

As for \eqref{III2}, we may assume $\tau_-(x,\bar{\zeta_1})\leq\tau_-(x,\zeta_2)$ without lost of generality and adopt the notations in the proof of Proposition  \ref{Geotheta}  after replacing $\zeta_1$ by $\bar{\zeta_1}$ in the expression.  Then,

\begin{equation}\begin{split}
&| III(x,\bar{\zeta_1})-III(x,\zeta_2)|\leq\\&\quad\left|\int_0^{\frac{|\overline{P_1x}|}{|\zeta_2|}}e^{-\nu(|\zeta_2|) s}\big[G(x-s\bar{\zeta_1},\bar{\zeta_1})-G(x-s\zeta_2,\zeta_2)\big]ds\right|\\&\quad+\left|\int_{\tau_-(x,\bar{\zeta_1})}^{\tau_-(x,{\zeta_2})}e^{-\nu(|\zeta_2|) s}G(x-s{\zeta_2},\zeta)ds\right|\\ &=:D_d+R_m.
\end{split}\end{equation}

When $|x-y|\geq \frac{d_0}2$, the desired estimate for $R_m$ holds similar to \eqref{Bdistanceest}.
When $|x-y|<\frac{d_0}2$, we use Proposition  \ref{Geotheta} and obtain
\begin{equation}\begin{split}
|R_m|&\leq C\Vert f\Vert_{L^\infty_{x,\zeta}}(1+d_0^{-1})\theta|\zeta_2|^{-1} e^{-\nu_0\frac{d_0}{2|\zeta_2|}}\\&\leq C\Vert f\Vert_{L^\infty_{x,\zeta}}(1+d_0^{-1})|\bar{\zeta_1}-\zeta_2||\zeta_2|^{-2} e^{-\nu_0\frac{d_0}{2|\zeta_2|}}\\& \leq C\Vert f\Vert_{L^\infty_{x,\zeta}}(1+d_0^{-1})^3|\bar{\zeta_1}-\zeta_2|\\& \leq C\Vert f\Vert_{L^\infty_{x,\zeta}}(1+d_0^{-1})^3\left(|\bar{\zeta_1}-\zeta_1|+|\zeta_1-\zeta_2|\right)\\& \leq C\Vert f\Vert_{L^\infty_{x,\zeta}}(1+d_0^{-1})^3|\zeta_1-\zeta_2|.\end{split}\end{equation}For $D_d$, we will subtract and add one term and then apply Lemma \ref{MixLemma} and Proposition \ref{Gzetadiff}:

\begin{equation}\label{DDestimate}\begin{split}
D_d\leq &\left|\int_0^{\frac{|\overline{P_1x}|}{|\zeta_2|}}e^{-\nu(|\zeta_2|) s}\big(G(x-s\bar{\zeta_1},\bar{\zeta_1})-G(x-s\zeta_2,\bar{\zeta_1})\big)ds\right|
\\&+\left|\int_0^{\frac{|\overline{P_1x}|}{|\zeta_2|}}e^{-\nu(|\zeta_2|) s}\big(G(x-s\zeta_2,\bar{\zeta_1})-G(x-s\zeta_2,\zeta_2)\big)ds\right|
\\\leq &C\Vert f\Vert_{L^{\infty}_{x,\zeta}}\left|\int_0^{\frac{|\overline{P_1x}|}{|\zeta_2|}}e^{-\nu(|\zeta_2| s}\big(({|\bar{\zeta_1}-\zeta_2|}s)^{\frac12}+|\bar{\zeta_1}-\zeta_2|\big)ds\right| \\ \leq &C\Vert f\Vert_{L^{\infty}_{x,\zeta}}|{\zeta_1}-\zeta_2|^{\frac12}.\end{split}\end{equation}Now, we proceed to  proving \eqref{III3}.  we assume $|\bar{\zeta_1}|
\leq |\zeta_1|$. We write\begin{equation}
\begin{split}
&|III(x,\bar{\zeta_1})-III(x,{\zeta_1})|\leq\\&\left|\int_{0}^{\tau_-({x,\zeta_1})} e^{-\nu(|\zeta_2|)s}G(x-s\bar{\zeta}_1,\bar{\zeta}_1)- e^{-\nu(|\zeta_1|)s}G(x-s\zeta_1,\zeta_1)ds\right|\\&+\left|\int_{\tau_-({x,\zeta_1})} ^{\tau_-({x,\bar{\zeta}_1})}e^{-\nu(|\zeta_2|)s}G(x-s\bar{\zeta}_1,\bar{\zeta}_1)ds\right|\\&=:D_v+R_v.\end{split}
\end{equation} 
We discuss $R_v$ in two cases. When $|\zeta_1|>2|\zeta_1-\zeta_2|^{\frac12}$, we notice that $|\zeta_2|>|\zeta_1-\zeta_2|^{\frac12}$. Then,
\begin{equation}\begin{split}
R_v&\leq C\Vert f \Vert_{L^\infty_{x,\zeta}}\frac{|\zeta_1|-|\zeta_2|}{|\zeta_1||\zeta_2|}e^{-\frac{\nu_0d_0}{|\zeta_1|}}\\&\leq Cd_0^{-1}\Vert f \Vert_{L^\infty_{x,\zeta}}\frac{|\zeta_1|-|\zeta_2|}{|\zeta_2|}\\&\leq  Cd_0^{-1}\Vert f \Vert_{L^\infty_{x,\zeta}}{|\zeta_1-\zeta_2|^\frac12}.
\end{split}\end{equation} 
When $|\zeta_1|\leq  2|\zeta_1-\zeta_2|^\frac12$, $|\zeta_2|\leq  2|\zeta_1-\zeta_2|^\frac12$ because we assume $|\zeta_2|\leq|\zeta_1|$.
Therefore, \begin{equation}\begin{split}
R_v&\leq C\Vert f \Vert_{L^\infty_{x,\zeta}}\int_{\tau_-({x,\zeta_1})} ^{\tau_-({x,\bar{\zeta}_1})}e^{-\nu_0s}ds\\&\leq C\Vert f \Vert_{L^\infty_{x,\zeta}}\left(e^{-\frac{\nu_0 d_0}{|\zeta_1|}}+e^{-\frac{\nu_0 d_0}{|\zeta_2|}}\right)\\&\leq C\Vert f \Vert_{L^\infty_{x,\zeta}}\left(\frac{|\zeta_1|}{\nu_0d_0}+\frac{|\zeta_2|}{\nu_0d_0}\right)\\&\leq  Cd_0^{-1}\Vert f \Vert_{L^\infty_{x,\zeta}}{|\zeta_1-\zeta_2|^\frac12}.
\end{split}\end{equation} 
As for $D_v$,
\begin{equation}\begin{split}
|D_v|&\leq \int_0^{\frac{|\overline{xP_1}|}{|\zeta_1|}}\left|e^{-\nu(|\zeta_2|) s}-e^{-\nu(|\zeta_1|) s}\right||G(x-s\bar{\zeta_1},\bar{\zeta_1})|ds\\&\quad+\int_0^{\frac{|\overline{xP_1}|}{|\zeta_1|}}e^{-\nu(|\zeta_1|) s}\left|G(x-s\bar{\zeta_1},\bar{\zeta_1})-G(x-s{\zeta_1},\bar{\zeta_1})\right|ds
\\&\quad+\int_0^{\frac{|\overline{xP_1}|}{|\zeta_1|}}e^{-\nu(|\zeta_1|) s}\left|G(x-s{\zeta_1},\bar{\zeta_1})-G(x-s{\zeta_1}, {\zeta_1})\right|ds\\&=:D_v^1+D_v^2+D_v^3.
\end{split}\end{equation}
We can deal with $D_v^2$ and $D_v^3$ similar to \eqref{DDestimate}.
To deal with $D_v^1$, we need to discuss the derivative of $\nu$:
\begin{equation}\label{nuderivative}\begin{split}
|\bigtriangledown\nu_{\gamma}(|\zeta|)|&=\left|\beta_0\int_{\mathbb{R}^3}e^{-|\eta|^2}|\eta-\zeta|^{\gamma-1}\frac{\eta-\zeta}{|\eta-\zeta|}d\eta\right|\\&\leq C(1+|\zeta|)^{\gamma-1}.\end{split}\end{equation} 

Applying the Mean Value Theorem, there exists $z$ between $|\zeta_1|$ and $|\zeta_2|$ such that

\begin{equation}\begin{split}
D_v^1&\leq C \int_0^{\frac{|\overline{xP_1}|}{|\zeta_1|}}\left||\zeta_2|-|\zeta_1|\right|e^{-\nu(z) s} (1+z)^{\gamma-1}|G(x-s\bar{\zeta_1},\bar{\zeta_1})|ds\\&\leq C\Vert f \Vert_{L^\infty_{x,\zeta}}{|\zeta_1-\zeta_2|}.
\end{split}\end{equation}Therefore, we prove the regularity of the $III$.
\end{proof}

\section{Preservation of the regularity from the boundary\label{Boundaryreg}}
In this section, we will prove that the damped transport equation preserves the regularity  of the boundary data and so does $II$. 

\begin{prop}\label{Iesttheta} 
Let $0<\sigma<\frac12$ and $\Omega$ be a bounded $C^1$ convex domain.
Suppose  $f\in {L^{\infty}_{x,\zeta}}$ satisfies \begin{equation}
|f(X,\zeta)-f(Y, \zeta')|\leq M(|X-Y|^2+|\zeta-\zeta'|^2)^{\frac{\sigma}{2}}
\end{equation}
 for any  $(X,\zeta),\ (Y,\zeta')\in\Gamma_-$. 
Then, there exists a constant $C_9$ depending only on $\Vert f\Vert _{L^{\infty}_{x,\zeta}}$, $M$, $\Omega$, and the potential such that,  $I$ from \eqref{123def} satisfies

\begin{align}\label{Ixestimate}|I(x,\zeta)-I(y,\zeta)|&\leq C_9(1+d_0^{-1})^{2\sigma}|x-y|^{\sigma},\\\label{Ithetaest}
|I(x,\zeta_1)-I(x,\zeta_2)|&\leq C_9(1+d_0^{-1})^2|\zeta_1-\zeta_2|^{\sigma}.
\end{align}As a consequence, 

\begin{equation}
|I(x,\zeta_1)-I(y, \zeta_2)|\leq 2C_9(1+d_0^{-1})^2(|x-y|^2+|\zeta_1-\zeta_2|^2)^{\frac{\sigma}{2}}.
\end{equation}
\end{prop}
\begin{proof} We can see that  $I$ is bounded and therefore the proposition holds if $|x-y|>1$ or $|\zeta_1-\zeta_2|>1$. We only need to discuss when $|x-y|\leq1$ and $|\zeta_1-\zeta_2|\leq1$.  

We start with \eqref{Ixestimate}. Notice that here we are picky about the growth rate with respect to $d_0$ because that it is needed for the proof of regularity of $II$.
Notice that if $d_0<|x-y|$, then $1<d_0^{-1}|x-y|$ and therefore the inequality holds. Hence, we only need to discuss the case $|x-y|<d_0$.
\begin{equation}\begin{split}
|I(x,\zeta)-I(y,\zeta)|&=\left|f(p(x,\zeta),\zeta)e^{-\nu(|\zeta|)\tau_-(x,\zeta)}-f(p(y,\zeta),\zeta)e^{-\nu(|\zeta|)\tau_-(y,\zeta)}\right|\\&\leq\left|f(p(x,\zeta),\zeta)-f(p(y,\zeta),\zeta)\right|e^{-\nu(|\zeta|)\tau_-(x,\zeta)}\\&\quad+|f(p(y,\zeta),\zeta)|\left|e^{-\nu(|\zeta|)\tau_-(x,\zeta)}-e^{-\nu(|\zeta|)\tau_-(y,\zeta)}\right|\\&=:I_a+I_b.\end{split}\end{equation}
From \eqref{geometry1} and our assumption, we have
\begin{equation}
|I_a|\leq CM(1+d_0^{-1})^{\sigma}|x-y|^{\sigma}.\end{equation}

To deal with $I_b$, we need to discuss the following two cases. 
If $|\zeta|\geq(1+ d_0^{-1})^{1-2\sigma}|x-y|^{1-\sigma}$, then
\begin{equation}\begin{split}
&|e^{-\nu(|\zeta|)\tau_-(x,\zeta)}-e^{-\nu(|\zeta|)\tau_-(y,\zeta)}|\leq C(1+|\zeta|)^\gamma|\tau_-(x,\zeta)-\tau_-(y,\zeta)|\\&\leq C((1+d_0^{-1})(1+|\zeta|)^\gamma\frac{|x-y|}{|\zeta|}\leq C\left\{\begin{array}{ll}(1+d_0^{-1})|x-y|,& |\zeta|>1\\(1+d_0^{-1})^{2\sigma}|x-y|^{\sigma}, & |\zeta|\leq 1\end{array}\right.\\&\leq C(1+d_0^{-1})^{2\sigma}|x-y|^{\sigma}.\end{split}
\end{equation}
Notice that $0\leq\gamma<1$ and $|x-y|<d_0$ were used in the above inequalities.
For the  case $|\zeta|<(1+ d_0^{-1})^{1-2\sigma}|x-y|^{1-\sigma}$,
\begin{equation}\label{zetasmallest1}\begin{split}
&|e^{-\nu\tau_-(x,\zeta)}-e^{-\nu\tau_-(y,\zeta)}|\leq 2e^{-\nu_0\frac{d_0}{|\zeta|}}\\&\leq2\left(\frac{|\zeta|}{\nu_0d_0}\right)^\frac{\sigma}{1-\sigma}\left(\frac{\nu_0d_0}{|\zeta|}\right)^\frac{\sigma}{1-\sigma}e^{-\nu_0\frac{d_0}{|\zeta|}}\\&\leq C\left(\frac{|\zeta|}{\nu_0d_0}\right)^\frac{\sigma}{1-\sigma}\leq C(1+d_0^{-1})^{2\sigma}|x-y|^\sigma.\end{split}
\end{equation}

For \eqref{Ithetaest}, we adopt the notation in the proof of Proposition \ref{Geotheta}. We introduce  $\bar{\zeta_1}=\frac{|\zeta_2|}{|\zeta_1|}\zeta_1$ and write
\begin{equation}\begin{split}
\left|I(x,\zeta_1)-I(x,\zeta_2)\right|&\leq\left|I(x,\zeta_1)-I(x,\bar{\zeta_1})\right|+\left|I(x,\bar{\zeta_1})-I(x,\zeta_2)\right|\\&=:\Delta I_r+\Delta I_\theta.\end{split}
\end{equation} 
For the former part,

\begin{equation}\begin{split}
&\Delta I_r=\left|f(P_1,\zeta_1)e^{-\nu(|\zeta_1|)\frac{|\overline{xP_1}|}{|\zeta_1|}}-f(P_1,\bar{\zeta_1})e^{-\nu(|\zeta_2|)\frac{|\overline{xP_1}|}{|\zeta_2|}}\right|
\\&\leq  \left|[f(P_1,\zeta_1)-f(P_1,\bar{\zeta_1})]e^{-\nu(|\zeta_1|)\frac{|\overline{xP_1}|}{|\zeta_1|}}\right|\\&\quad+\left|f(P_1,\bar{\zeta_1})\left[e^{-\nu(|\zeta_1|)\frac{|\overline{xP_1}|}{|\zeta_1|}}-e^{-\nu(|\zeta_2|)\frac{|\overline{xP_1}|}{|\zeta_2|}}\right]\right|\\&=:\Delta I_r^1+\Delta I_r^2.
\end{split}
\end{equation}
We can obtain the desired bound for $\Delta I_r^1$ directly from the assumption. 
For $\Delta I_r^2$, we need to discuss in two cases.
First, we consider the case $|\zeta_1|\leq 2\big||\zeta_1|-|\zeta_2|\big|^{\frac{1-\sigma}2}$. Notice that $|\zeta_2|\leq 3\big||\zeta_1|-|\zeta_2|\big|^{\frac{1-\sigma}2}$.  Similar to \eqref{zetasmallest1}, we have

 \begin{equation}\label{zeta1small} \begin{split}&|e^{-\nu(|\zeta_1|)\frac{|\overline{xP_1}|}{|\zeta_1|}}|+|e^{-\nu(|\zeta_2|)\frac{|\overline{xP_1}|}{|\zeta_2|}}|\leq e^{-\frac{\nu_0d_0}{|\zeta_1|}}+e^{-\frac{\nu_0d_0}{|\zeta_2|}}\\&
 \leq C\left(\left| \frac{|\zeta_1|}{\nu_0d_0}\right|^{\frac{2\sigma}{1-\sigma}}+\left| \frac{|\zeta_2|}{\nu_0d_0}\right|^{\frac{2\sigma}{1-\sigma}}\right)\\&\leq C(\frac1{\nu_0d_0})^{\frac{2\sigma}{1-\sigma}} \big||\zeta_1|-|\zeta_2|\big|^{\sigma}\\&\leq C(1+d_0^{-1})^2|\zeta_1-\zeta_2|^{\sigma}.\end{split}\end{equation}
  
 When $|\zeta_1|>2\big||\zeta_1|-|\zeta_2||\big|^{\frac{1-\sigma}2}$,  we  have $|\zeta_2|> \big||\zeta_1|-|\zeta_2||\big|^{\frac{1-\sigma}2} $.
% Recall that \begin{equation}
 %\nu(\eta)=2^{-\frac32}[e^{-\eta^2}+(2\eta+\frac1\eta)\int^{\eta}_0e^{-s^2}ds]
 %\end{equation}
 By the Mean Value Theorem, there exists  $\eta'$ between $|\zeta_1|$ and $|\zeta_2|$, therefore $\eta'>\big||\zeta_1|-|\zeta_2||\big|^{\frac{1-\sigma}2}$, such that, 
 \begin{equation}\begin{split}
&| \Delta I_r^2|=   \left|f(P_1,\bar{\zeta_1})|\overline{xP_1}|(|\zeta_1|-|\zeta_2|)\left.e^{-\nu(\eta')\frac{|\overline{xP_1}|}{\eta'}}\frac{d}{d\eta}\big(\frac{\nu(\eta)}{\eta}\big)\right|_{\eta=\eta'}\right|\\& \leq \big||\zeta_1|-|\zeta_2|\big|\left\{\begin{array}{ll}\frac1{{\eta'}^2},& 0 <\eta'<1,
\\ \frac1{{\eta'}^{2-\gamma}},&\eta'\geq1\end{array}\right.
\ 
\\&\leq C \big||\zeta_1|-|\zeta_2|\big|^{\sigma}\leq  C |\zeta_1-\zeta_2|^{\sigma}.
\end{split}
 \end{equation}
 Notice that we used the estimate \eqref{nuderivative} and the fact $|\zeta_1-\zeta_2|\leq 1$.
 
 Now, we   proceed to $\Delta I_\theta$. 
  \begin{equation}\begin{split}
 &|\Delta I_\theta|=\left|f(P_1,\bar{\zeta_1})e^{-\nu(|\zeta_2|)\frac{|\overline{xP_1}|}{|\zeta_2|}}-f(P_2,\zeta_2)e^{-\nu(|\zeta_2|)\frac{|\overline{xP_2}|}{|\zeta_2|}}\right| \\ &
 \leq \big|f(P_1,\bar{\zeta_1})-f(P_2,\zeta_2)\big|e^{-\nu(|\zeta_2|)\frac{|\overline{xP_1}|}{|\zeta_2|}}\\&\quad+\left|f(P_2,\zeta_2)\right|\left|e^{-\nu(|\zeta_2|)\frac{|\overline{xP_1}|}{|\zeta_2|}}-e^{-\nu(|\zeta_2|)\frac{|\overline{xP_2}|}{|\zeta_2|}}\right|\\&=:\Delta I_\theta^1+\Delta I_\theta^2.
 \end{split}\end{equation}
 
 Applying Proposition \ref{Geotheta} and our assumption, we have
 \begin{equation}\begin{split}
 & |\Delta I_\theta^1|\leq C\big[(1+d_0^{-1})^{\sigma}|\theta|^{\sigma}e^{-\nu_0\frac{d_0}{|\zeta_2|}}+|\bar{\zeta_1}-\zeta_2|^{\sigma}\big]\\ &\leq C\big[(1+d_0^{-1})^{\sigma}\left|\frac{|\bar{\zeta_1}-\zeta_2|}{|\zeta_2|}
\right|^{\sigma}e^{-\nu_0\frac{d_0}{|\zeta_2|}}+|\bar{\zeta_1}-\zeta_2|^{\sigma}\big]\\&\leq C(1+d_0^{-1})^{2\sigma}|{\zeta_1}-\zeta_2|^{\sigma}.\end{split}
 \end{equation}
 For $\Delta I_\theta^2$, we again break it into two cases: $|\zeta_2|\leq |\bar{\zeta_1}-\zeta_2|^{\frac{1-\sigma}{2}}$ and $|\zeta_2|> |\bar{\zeta_1}-\zeta_2|^{\frac{1-\sigma}{2}}$.
 We mimic \eqref{zeta1small} and obtain the desired bound for the former case. 
 For the latter one, applying the Mean Value Theorem    and Proposition \ref{Geotheta}, we have
 \begin{equation}\begin{split}&
 |e^{-\nu(|\zeta_2|)\frac{|\overline{xP_1}|}{|\zeta_2|}}-e^{-\nu(|\zeta_2|)\frac{|\overline{xP_2}|}{|\zeta_2|}}|\leq \frac{\nu(|\zeta_2|)}{|\zeta_2|}\big||\overline{xP_1}|-|\overline{xP_2}|\big|\\&\leq C(1+d_0^{-1})\frac{\nu(|\zeta_2|)}{|\zeta_2|}\frac{|\bar{\zeta_1}-\zeta_2|}{|\zeta_2|}\leq  C(1+d_0^{-1})\left\{\begin{array}{ll}|\bar{\zeta_1}-\zeta_2|^{\sigma}, &  |\zeta_2|<1\\|\bar{\zeta_1}-\zeta_2|,&|\zeta_2|\geq1\end{array}\right.\\&\leq C(1+d_0^{-1})|\zeta_1-\zeta_2|^{\sigma}.
 \end{split}\end{equation}
 Notice that the last inequality above holds because we are now discussing the case when $|\zeta_1-\zeta_2|\leq 1$. Hence, we complete the proof.
  
\end{proof}

Now we  only have $II$ left.

\begin{prop}
With the same assumption as Proposition \ref{Iesttheta} above, there exists a constant $C_{10}$ only depending on $\Vert f\Vert_{L^\infty_{x,\zeta}}$,  $M$,  $\sigma$, $\Omega$, and the potential such that, 
\begin{equation}
|II(x,\zeta_1)-II(y,\zeta_2)|\leq C_{10}(1+d_0^{-1})^{3}(|x-y|^2+|\zeta_1-\zeta_2|^2)^{\frac{\sigma}2}.
\end{equation} 
\end{prop}

\begin{proof}  Notice that the inequality holds when $|x-y|$ or $|\zeta_1-\zeta_2|$ are large compare with $d_0$. We only need the deal the case when both of them are small.  We can break the proof into the following three estimates:
 \begin{align}
&\label{II1holder}|II(x,\zeta)-II(y,\zeta)|\leq C(1+d_0^{-1})^{2}|x-y|^{\sigma},\\&\label{II2holder}|II(x,\bar{\zeta_1})-II(x,\zeta_2)|\leq C(1+d_0^{-1})^{3}|\zeta_1-\zeta_2|^{\sigma},\\\label{II3holder}
&|II(x,\bar{\zeta_1})-II(x,\zeta_1)|\leq C(1+d_0^{-1})^{2}|\zeta_1-\zeta_2|^{\sigma}.
\end{align}   In the proof of \eqref{II1holder}, we may assume $\tau_-(y,\zeta)\geq\tau_-(x,\zeta)$. 
We will apply estimate  \eqref{Ixestimate} in Proposition \ref{Iesttheta} in our proof. Notice that  \eqref{Ixestimate}  depends on the distance and breaks down  at  
boundary. However, thank to the help of the integration over $s$ in $II$, we still can have the desired estimate.  To carry out the analysis, we fist dyadically  decompose $\Omega$ as fallows. Let ${M_n}=\{x\in\Omega|d(x,\partial{\Omega})\leq d_02^{-n}\}$ for integer $n\geq 1$. Let $D_0=\Omega\setminus M_1$ and $D_n=M_n\setminus M_{n+1}$ for $n\geq1$. Let $N$ be the integer  such that $8|x-y|>2^{-N}d_0\geq 4|x-y|$. $D_n$ for $0 \leq n\leq N$ together with $M_{N+1}$ form a decomposition of $\Omega$. Let $X_0=p(x,\zeta)$ and $Y_0=p(y,\zeta)$. Let  $Z_n$ be the intersection of $\overline{xX_0}$ with
$\partial M_n\setminus \partial \Omega$ for $ n\geq 1$ and $Z_0=x$.  Applying Proposition \ref{distanceXzest}, we have\begin{equation}
|\overline {XZ_n}|\leq R d_0^{-1}2^{-n} d_0\leq R 2^{-n}.
\end{equation}
For $n\geq1$,
\begin{equation}\begin{split}
&\Delta II_d^n:=\left|\int_{\frac{|\overline{Z_nx}|}{|\zeta|}}^{\frac{|\overline{Z_{n+1}x}|}{|\zeta|}}e^{-\nu s }\int_{\mathbb{R}^3}k(\zeta,\zeta')[I(x-s\zeta,\zeta')-I(y-s\zeta,\zeta')]d\zeta'ds\right|\\&\leq C\left|\int_{\frac{|\overline{Z_nx}|}{|\zeta|}}^{\frac{|\overline{Z_{n+1}x}|}{|\zeta|}}e^{-\nu s }\int_{\mathbb{R}^3}k(\zeta,\zeta')(1+2^{n+2}d_0^{-1})^{2\sigma}|x-y|^{\sigma}d\zeta'ds\right|\\&\leq C(1+2^{n+2}d_0^{-1})^{2\sigma}|x-y|^{\sigma}\left|\int_{\frac{|\overline{Z_nx}|}{|\zeta|}}^{\frac{|\overline{Z_{n+1}x}|}{|\zeta|}}e^{-\nu s }ds\right|
\\&\leq C(1+2^{n+2}d_0^{-1})^{2\sigma}|x-y|^{\sigma}R2^{-n}|\zeta|^{-1}e^{-\frac{\nu_0d_0}{2|\zeta|}}
\\&\leq C2^{-(1-2\sigma)n}(1+d_0^{-1})^{1+2\sigma}   |x-y|^{\sigma}.
\end{split}\end{equation}
Notice that $1-2\sigma>0 $ and therefore \begin{equation}\sum_{n=1}^NII_d^n\leq C(1+d_0^{-1})^{1+2\sigma}   |x-y|^{\sigma}.
 \end{equation}
 For  the segment within $D_0$,
 \begin{equation}\label{D01est}\begin{split}
&\Delta II_d^0:=\left|\int^{\frac{|\overline{Z_1x}|}{|\zeta|}}_0e^{-\nu(|\zeta|) s }\int_{\mathbb{R}^3}k(\zeta,\zeta')[I(x-s\zeta,\zeta')-I(y-s\zeta,\zeta')]d\zeta'ds\right|\\&\leq C(1+4d_0^{-1})^{2\sigma}|x-y|^{\sigma}\int^{\frac{|\overline{Z_1x}|}{|\zeta|}}_0e^{-\nu(|\zeta|) s }\int_{\mathbb{R}^3}|k(\zeta,\zeta')|d\zeta'ds\\&\leq C(1+d_0^{-1})^{2\sigma}|x-y|^{\sigma}\left(\frac1{\nu(|\zeta|)}-\frac1{\nu(|\zeta|)}e^{-\nu(|\zeta|)\frac{|\overline{Z_1x}|}{|\zeta|}}\right)
\\&\leq C(1+d_0^{-1})^{2\sigma}|x-y|^{\sigma}.\end{split}\end{equation}
 
 For the remaining part, we have
 \begin{equation}\begin{split}
 &\left|\int_{\frac{|\overline{Z_{N+1}x}|}{|\zeta|}}^{\tau_-(x,\zeta)}e^{-\nu s }\int_{\mathbb{R}^3}k(\zeta,\zeta')I(x-s\zeta,\zeta')d\zeta'ds\right|\\&\quad+\left|\int_{\frac{|\overline{Z_{N+1}x}|}{|\zeta|}}^{\tau_-(y,\zeta)}e^{-\nu s }\int_{\mathbb{R}^3}k(\zeta,\zeta')I(y-s\zeta,\zeta')d\zeta'ds\right| \\&\leq C(1+d_0^{-1})\frac{|x-y|}{|\zeta|}e^{-\frac{\nu_0d_0}{2|\zeta|}}\\&\leq
 C(1+d_0^{-1})^2|x-y|,\end{split}
 \end{equation} where the definition of $N$ and Proposition \ref{G12} were used.  Hence, we proved \eqref{II1holder}. 
 
 Now, we proceed to $\eqref{II2holder}$.  Let $\theta$ be the angle between $\bar{\zeta_1} $ and $\zeta_2$. Without lost of generality, we may assume $\tau_-(x, \bar{\zeta_1})\leq \tau_-(x,\zeta_2)$. Let $P_1=p(x,\bar{\zeta_1})$ and $P_2=p(x,\zeta_2)$. Again, we only need to prove the case when $|\bar{\zeta_1}-\zeta_2|\leq 1$. Let $\theta$ be the angle between $\zeta_1$ and $\zeta_2$. We adopt the same dyadic decomposition above and stop at $N'$ such that 
 $8|\overline{P_1x}|\theta>2^{-N'}d_0\geq 4|\overline{P_1x}|\theta$.  Let $S_n$ be the intersection between $\overline{P_1x}$ and $\partial M_n\setminus \partial \Omega$. 

 \begin{equation}\begin{split}
 &\Delta II^n:=\left|\int_{\frac{|\overline{xS_n}|}{|\zeta_2|}}^{\frac{|\overline{xS_{n+1}}|}{|\zeta_2|}} e^{-\nu s }\int_{\mathbb{R}^3}k(\bar{\zeta_1},\zeta')I(x-s\bar{\zeta_1},\zeta')-k(\zeta_2,\zeta')I(x-s\zeta_2,\zeta')d\zeta'ds\right|\\&\leq \left|\int_{\frac{|\overline{xS_n}|}{|\zeta_2|}}^{\frac{|\overline{xS_{n+1}}|}{|\zeta_2|}} e^{-\nu s }\int_{\mathbb{R}^3}k(\bar{\zeta_1},\zeta')[I(x-s\bar{\zeta_1},\zeta')-I(x-s\zeta_2,\zeta')]d\zeta'ds\right|\\&\quad+\left|\int_{\frac{|\overline{xS_n}|}{|\zeta_2|}}^{\frac{|\overline{xS_{n+1}}|}{|\zeta_2|}} e^{-\nu s }\int_{\mathbb{R}^3}[k(\bar{\zeta_1},\zeta')-k(\zeta_2,\zeta')]I(x-s\zeta_2,\zeta')d\zeta'ds\right|\\&=:DI^n+DK^n.
 \end{split}
 \end{equation}
 $DI^n$ can be bounded by a similar argument:
 
 \begin{equation}\begin{split}
&DI^n\leq C\left|\int_{\frac{|\overline{xS_n}|}{|\zeta_2|}}^{\frac{|\overline{xS_{n+1}}|}{|\zeta_2|}} e^{-\nu s }\int_{\mathbb{R}^3}k(\bar{\zeta_1},\zeta')(1+2^{n+2}d_0^{-1})^{2\sigma}(\frac{|\bar{\zeta_1}-\zeta_2|}{|\zeta_2|})^{\sigma}d\zeta'ds\right|\\&\leq C(1+2^{n+2}d_0^{-1})^{2\sigma}(\frac{|\bar{\zeta_1}-\zeta_2|}{|\zeta_2|})^{\sigma}\left|\int_{\frac{|\overline{xS_n}|}{|\zeta_2|}}^{\frac{|\overline{xS_{n+1}}|}{|\zeta_2|}} e^{-\nu s }ds\right|
\\&\leq C(1+2^{n+2}d_0^{-1})^{2\sigma}(\frac{|\bar{\zeta_1}-\zeta_2|}{|\zeta_2|})^{\sigma}R2^{-n}|\zeta_2|^{-1}e^{-\frac{\nu_0d_0}{2|\zeta_2|}}
\\&\leq C2^{-(1-2\sigma)n}(1+d_0^{-1})^{1+3\sigma}  |\bar{\zeta_1}-\zeta_2|^{\sigma}.
\end{split}\end{equation}
Therefore, 
\begin{equation}
\sum_1^{N'}DI^n\leq C(1+d_0^{-1})^{\frac52}   |\bar{\zeta_1}-\zeta_2|^{\sigma}.
\end{equation}
For $DK^n$ part, from the assumption, we observe that $I(x,\zeta)$ is bounded.
Therefore, we have
\begin{equation}
\left|\int_{\mathbb{R}^3}[k(\bar{\zeta_1},\zeta')-k(\zeta_2,\zeta')]I(x-s\zeta_2,\zeta')d\zeta'\right|\leq C|\bar{\zeta_1}-\zeta_2|.
\end{equation}
Hence,\begin{equation}\begin{split}
&DK^n\leq  C|\bar{\zeta_1}-\zeta_2|\left|\int_{\frac{|\overline{xS_n}|}{|\zeta_2|}}^{\frac{|\overline{xS_{n+1}}|}{|\zeta_2|}} e^{-\nu s }ds\right|\\&\leq C|\bar{\zeta_1}-\zeta_2| \frac{R2^{-n}}{|\zeta_2|}e^{-\frac{\nu_0d_0}{2|\zeta_2|}}\\&\leq C(1+d_0^{-1})2^{-n}|\bar{\zeta_1}-\zeta_2|. \end{split}\end{equation}
As a result,
\begin{equation}\sum_{n=1}^{N'} DK_n\leq  C(1+d_0^{-1})|\bar{\zeta_1}-\zeta_2|.\end{equation}
For the segment within $D_0$,
\begin{equation}\label{D02est}\begin{split}
 &\Delta II^0:=\left|\int_{0}^{\frac{|\overline{xS_1}|}{|\zeta_2|}} e^{-\nu s }\int_{\mathbb{R}^3}k(\bar{\zeta_1},\zeta')I(x-s\bar{\zeta_1},\zeta')-k(\zeta_2,\zeta')I(x-s\zeta_2,\zeta')d\zeta'ds\right|\\&\leq C\int_{0}^{\frac{|\overline{xS_1}|}{|\zeta_2|}} e^{-\nu s }\left[(1+d_0^{-1})^{2\sigma}(s|\bar{\zeta_1}-\zeta_2|)^{\sigma}+|\bar{\zeta_1}-\zeta_2|\right]ds\\&\leq C(1+d_0^{-1})^{2\sigma}|\bar{\zeta_1}-\zeta_2|^{\sigma}.\end{split}\end{equation}
 Similarly, we can also establish bounds for the remaining part because the domain of integration is small and distance from $x$ is big.  Hence, we proved \eqref{II2holder}. 
 
 Finally  we are going to prove \eqref{II3holder}.  We also need the dyadic decomposition used above. Let $\bar{N}$ be the integer such that
 \begin{equation}
 4|\zeta_2-\zeta_1|\geq   2^{-\bar{N}}d_0\geq 2|\zeta_2-\zeta_1|.
 \end{equation} Without lost generality, we may assume $|\bar{\zeta_1}|\leq |\zeta_1|$. Let $S_0=x$.
 For $1\leq n\leq \bar{N}$,
 \begin{equation}\begin{split}
 &DV^{n}:=\\&\left|\int_{\frac{|\overline{S_nx}|}{|\zeta_1|}}^{\frac{|\overline{S_{n+1}x}|}{|\zeta_1|}}\int_{\mathbb{R}^3}e^{-\nu(|\zeta_2|)s}k(\bar{\zeta_1},\zeta')I(x-s\bar{\zeta_1},\zeta')-e^{-\nu(|\zeta_1|)s}k({\zeta_1},\zeta')I(x-s{\zeta_1},\zeta')d\zeta'ds\right|
 \\&\leq\Big|\int_{\frac{|\overline{S_nx}|}{|\zeta_1|}}^{\frac{|\overline{S_{n+1}x}|}{|\zeta_1|}}\int_{\mathbb{R}^3}\left[e^{-\nu(|\zeta_2|)s}-e^{-\nu(|\zeta_1|)s}\right]k(\bar{\zeta_1},\zeta')I(x-s\bar{\zeta_1},\zeta')
 \\&\quad+e^{-\nu(|\zeta_1|)s}\left[k(\bar{\zeta_1},\zeta')-k({\zeta_1},\zeta')\right]I(x-s\bar{\zeta_1},\zeta')\\&\quad+e^{-\nu(|\zeta_1|)s}k({\zeta_1},\zeta')\left[I(x-s\bar{\zeta_1},\zeta')-I(x-s{\zeta_1},\zeta')\right]d\zeta'ds\Big|\\&\leq C\left[(1+d_0^{-1})|\zeta_2-\zeta_1|2^{-n}+(1+d_0^{-1})^{1+2\sigma}|\zeta_2-\zeta_1|^{\sigma}2^{-(1-2\sigma)n}\right].
\end{split}\end{equation}
Hence,  \begin{equation}\sum_{n=1}^{\bar{N}}DV^n\leq C(1+d_0^{-1})^{1+2\sigma}||\zeta_2-\zeta_1|^{\sigma}.
\end{equation}
We can bound the contribution from the  segment within $D_0$ similar to \eqref{D01est} and \eqref{D02est}.

For the remaining part, we define

\begin{equation}\begin{split}
&RV:=\\&\left|\int_{\frac{|\overline{S_{\bar{N}+1}x}|}{|\zeta_1|}}^{\frac{|\overline{P_1x}|}{|\zeta_2|}} \int_{\mathbb{R}^3}e^{-\nu(|\zeta_2|)s}k(\bar{\zeta_1},\zeta')I(x-s\bar{\zeta_1},\zeta')-e^{-\nu(|\zeta_1|)s}k(\zeta_2,\zeta')I(x-s\zeta_2,\zeta')d\zeta'ds\right|.\end{split}\end{equation}
When $|\zeta_1|>2\sqrt{|\zeta_1-\zeta_2|}$, $|\zeta_2|>\sqrt{|\zeta_1-\zeta_2|}$.
Therefore,

\begin{equation}\begin{split}
&RV\leq C\int_{\frac{|\overline{S_{\bar{N}+1}x}|}{|\zeta_1|}}^{\frac{|\overline{P_1x}|}{|\zeta_2|}} e^{-\nu_0s}ds\leq  C e^{-\frac{\nu_0d_0}{2|\zeta_1|}}\left[\frac{|\overline{P_1S_{\bar{N}+1}}|}{|\zeta_2|}+|\overline{S_{\bar{N}+1}x}|(\frac1{|\zeta_2|}-\frac1{|\zeta_1|})\right]\\&\leq  C e^{-\frac{\nu_0d_0}{2|\zeta_1|}}\left[(1+d_0^{-1})\frac{|\zeta_1-\zeta_2|}{|\zeta_2|}+\frac{\big||\zeta_1|-|\zeta_2|\big|}{|\zeta_2||\zeta_1|}\right]\\&\leq C(1+d_0^{-1})\frac{|\zeta_1-\zeta_2|}{|\zeta_2|}\leq C(1+d_0^{-1}){|\zeta_1-\zeta_2|^\frac12}.\end{split}\end{equation} 
 For the other case,  $2\sqrt{|\zeta_1-\zeta_2|}\geq|\zeta_1|\geq|\zeta_2|$.
 \begin{equation}\begin{split}
&RV\leq C\int_{\frac{|\overline{S_{\bar{N}+1}x}|}{|\zeta_1|}}^{\frac{|\overline{P_1x}|}{|\zeta_2|}} e^{-\nu_0s}ds\leq  C(e^{-\nu_0\frac{|\overline{S_{\bar{N}+1}x}|}{|\zeta_1|}}+e^ {-\nu_0\frac{|\overline{P_1x}|}{|\zeta_2|}}) \\&\leq Cd_0^{-1}(|\zeta_1|+|\zeta_2|)\leq Cd_0^{-1}\sqrt{|\zeta_1-\zeta_2|}.\end{split}\end{equation} 

Hence, we finished the proof.\end{proof}
 
  \section{Gaining of integrability \label{convolution}}
  We devote  this final section to the proof of the following proposition.
  \begin{prop}Suppose $f\in L^*_{x,\zeta}$ solves the stationary linearized Boltzmann equation \eqref{SBE} for gases with cutoff hard potential or cutoff Maxwellian gases, \eqref{IP},   and  there exists $\phi(\zeta)\in L^*_\zeta$ such that
  \begin{equation}
  |f(X,\eta)|<\phi(\zeta),
  \end{equation}
  for any $(X,\eta)\in \Gamma_-$.
  Then, $f\in L^\infty_xL^*_\zeta$.  
 
 Moreover, if  we further assume that  for a fix $0<\alpha<\frac12$ there exist $M$ such that
  \begin{equation}
f(X,\eta)-f(Y,\omega)|\leq M\left( |\eta-\omega|^2+|X-Y|^2\right)^{\frac{\sigma}{2}} \end{equation}
for any $(X,\eta)$, $(Y, \omega)\in\Gamma_-$,
then $f\in L^\infty_{x,\zeta}$.
  
  \end{prop}
To prove $f\in L^\infty_xL^*_\zeta$,
  we  first use the integral equation \eqref{inteq} and name the former term on the right hand side as $\tilde{g}(x,\zeta)$ and the latter as $\tilde{f}(x,\zeta)$.   From the assumption, we know that both $\Vert \tilde{g}\Vert_{L^1_xL^*_\zeta} $ and $\Vert \tilde{g}\Vert_{L^\infty_xL^*_\zeta} $ are bounded,  therefore, through interpolation, so is $\Vert \tilde{g}\Vert_{L^p_xL^*_\zeta} $   for any $p\in[1,\infty]$.
We perform zero extension for $f(x,\zeta)$ outside $\Omega$ and, with abuse of notation, still call it $f$.  We use $*$ to denote the convolution in space.
The following  estimate is to be proved.
\begin{lem}\label{fconvalution}Let $0<\alpha<1$.  $f$ and $\tilde{f}$ are as defined above. Then, there exists a constant $C_{11}$ only depending on the potential such that
\begin{equation}\label{Fconv}  
 \Vert \tilde{f}(x,\cdot)\Vert_{L^*_\zeta}^2\leq C_{11} (\frac{\chi_{(0,R)}(|x|)}{|x|^{2-\alpha}})* \Vert f(x,\cdot) \Vert_{L^*_\zeta}^2 ,
 \end{equation}
where $R$ is the diameter of $\Omega$ and $\chi_{(0,R)}$ is the characteristic function of $(0,R)$.\end{lem}
Let $1<p<\frac3{2-\alpha}$, $q, r \in [0,\infty]$.  Applying Young's inequality to the convolution above, we can conclude that, if \begin{equation}\frac2r+1=\frac1p+\frac2q,
 \end{equation}
 then $f\in L^q_xL^*_\zeta$ implies $\tilde{f}\in L^r_xL^*_\zeta$. 
 Starting from $q=2$, through finite times of iteration, we can obtain that $f\in L^\infty_xL^*_\zeta$.  Notice that the hard sphere case is barely missed. 
 
 The fact $f\in L^\infty_xL^*_\zeta$ together with H\"{o}lder continuity assumption on the boundary data  implies
 \begin{equation}|f(X,\zeta)|\leq C_{12}M^\frac{3}{3+2\sigma}\Vert f\Vert_{L^\infty_xL^*_\zeta}^{\frac{2\sigma}{3+2\sigma}},\end{equation} where $C_{12}=2(\frac{3}{4\pi\nu_0})^{\frac{\sigma}{3+2\sigma}}$ for any $(X,\zeta)\in \Gamma_-$.  Then, from the integral equation \eqref{inteq}, we can conclude $f\in L^\infty_{x,\zeta}$.
 
  In order to prove Lemma  \ref{fconvalution}, we shall first estimate the decay of $K(f)$:      \begin{prop} \label{Kdecay} For any $0\leq\gamma\leq1$, 
   \begin{equation}
   |K(f)| \leq C \Vert f\Vert_{L^*_\zeta}(1+|\zeta|)^{-\frac{3-\gamma}{2}}.
   \end{equation}
   
   The constant $C$ above may depend on the potential. 
  \end{prop}
  \begin{proof}
  
  \begin{equation}\begin{split}
  |K(f)(\zeta)|=&\left|\int_{\mathbb{R}^3}k(\zeta,\zeta_*)f(\zeta_*)d\zeta_*\right|\\=&\left(\int_{\mathbb{R}^3}|k(\zeta,\zeta_*)|^2\frac1{|\nu(\zeta_*)|}d\zeta_*\right)^\frac12\left(\int_{\mathbb{R}^3}|\nu(\zeta_*)||f(\zeta_*)|^2d\zeta_*\right)^\frac12\\ \leq&\Vert f\Vert_{L^*_\zeta} \left(\int_{\mathbb{R}^3}\frac{e^{-\frac14\left(|\zeta-\zeta_*|^2+(\frac{|\zeta|^2-|\zeta_*|^2}{|\zeta-\zeta_*|})^2\right)}}{|\zeta-\zeta_*|^2(1+|\zeta|+|\zeta_*|)^{2(1-\gamma)}(1+|\zeta_*|)^\gamma}d\zeta_*\right)^\frac12\\ \leq &C
  \Vert f\Vert_{L^*_\zeta} \left(\int_{|\zeta-\zeta_*|\leq \frac{|\zeta|}{2}}\frac{e^{-\frac14\left(|\zeta-\zeta_*|^2+(\frac{|\zeta|^2-|\zeta_*|^2}{|\zeta-\zeta_*|})^2\right)}}{|\zeta-\zeta_*|^2(1+|\zeta|)^{2(1-\gamma)}(1+\frac{|\zeta|}2)^\gamma}d\zeta_*\right.\\&\left.+\int_{|\zeta-\zeta_*|>\frac{|\zeta|}2}\frac{e^{-\frac14|\zeta-\zeta_*|^2}}{|\zeta-\zeta_*|^2(1+|\zeta|)^{2(1-\gamma)}}d\zeta_*\right)^\frac12\\\leq&  C\Vert f\Vert_{L^*_\zeta}\left[ (|1+|\zeta|)^{-(3-\gamma)}+|\zeta|^{-2}(1+|\zeta|)^{-(2-2\gamma)}\right]^\frac12.
  \end{split}\end{equation}
  Notice that $K(f)$ is also bounded.  Therefore, we conclude the proposition.
  \end{proof}
  We are ready to prove Lemma \ref{fconvalution}
  
 \begin{proof}[proof of Lemma \ref{fconvalution}]

 We will use the spherical coordinate such that
  \begin{equation}
  \zeta'=(\rho\cos\theta,\rho\sin\theta\cos\phi,\rho\sin\theta\sin\phi)\end{equation}
  and also  $r=s|\zeta|$ as we did in Section \ref{Mix}.  We have
 \begin{equation}\begin{split}
 &\Vert \tilde{f}(x,\cdot)\Vert_{L^*_\zeta}^2=\int_{\mathbb{R}^3}|\nu(\zeta)|\left|\int_0^{\tau_-(x,\zeta)}e^{-\nu s}K(f)(x-\zeta s,\zeta)ds\right|^2d\zeta\\&\leq 
 \int_{\mathbb{R}^3}|\nu(\zeta)|\left|\int_0^{\tau_-(x,\zeta)}e^{-2\nu s} s^{-\alpha}ds\right|\left|\int_0^{\tau_-(x,\zeta)}|K(f)(x-\zeta s,\zeta)|^2s^\alpha ds\right|d\zeta
 \\&\leq C\int_{\mathbb{R}^3}(1+|\zeta|)^{\gamma\alpha}\int_0^{|\overline{p(x,\zeta)x}|} \Vert f(x-r\hat{\zeta},\cdot)\Vert^2_{L^*_\zeta}(1+|\zeta|)^{-(3-\gamma)}r^\alpha\frac1{|\zeta|^{1+\alpha}}drd\zeta\\&\leq
 C\int_0^\infty (1+\rho)^{\gamma(1+\alpha)-3}\rho^{1-\alpha}\\&\quad\quad\cdot\int_0^\pi \int_0^{2\pi}\int_0^{|\overline{p(x,\zeta)x}|}\Vert f(x-r\hat{\zeta},\cdot)\Vert^2_{L^*_\zeta}r^\alpha dr \sin\theta d\phi d\theta d\rho\\&\leq C\int_{\Omega} \frac{\Vert f(y,\cdot)\Vert^2_{L^*_\zeta}}{|x-y|^{2-\alpha}}dy\int_0^\infty (1+\rho)^{\gamma(1+\alpha)-3}\rho^{1-\alpha}d\rho\\&\leq C\int_{\Omega} \frac{\Vert f(y,\cdot)\Vert^2_{L^*_\zeta}}{|x-y|^{2-\alpha}}dy.\end{split}\end{equation}
 Notice that the hard sphere case, $\gamma=1$, is barely missed.
 We make an zero extension of $f$ outside $\Omega$, and, with abuse of notation, still call it $f$.  Then, we conclude the lemma.  \end{proof}
 
\subsection*{Acknowledgments}
This research  is supported in part by JSPS KAKENHI grant number 15K17572.

\end{document}